\documentclass{amsart}
\usepackage{latexsym,amsmath,amssymb}
\usepackage{cite}
\newtheorem{thm}{Theorem}[section]
\newtheorem{cor}[thm]{Corollary}
\newtheorem{exam}[thm]{Example}

\newtheorem{prop}[thm]{Proposition}
\theoremstyle{definition}\newtheorem{defn}[thm]{Definition}
\theoremstyle{remark}

\numberwithin{equation}{section}

\begin{document}

\title[]
{On a class of m-Isometric and Quasi-m-isometric operators}

\author{\sc\bf Y. Estaremi, M. S. Al Ghafri and S. Shamsigamchi }
\address{\sc S. Shamsigamchi}
\address{Department of Mathematics, Payame Noor(PNU) university, Tehran, Iran.}
\email{S.shamsi@pnu.ac.ir}
\address{\sc Y. Estaremi}
\email{y.estaremi@gu.ac.ir}
\address{Department of Mathematics, Faculty of Sciences, Golestan University, Gorgan, Iran.}
\address{\sc M. S. Al Ghafri}
\email{mohammed.alghafri@utas.edu.om}
\address{Department of Mathematics, University of Technology and Applied Sciences, Rustaq {329}, Oman.}

\thanks{}

\thanks{}



\date{}

\dedicatory{}

\commby{}

\begin{abstract}
In this paper we characterize $m$-isometric and quasi-$m$-isometric weighted conditional type (WCT) operators on the Hilbert space $L^2(\mu)$. Also, we prove that the subclasses of $m$-isometric and quasi-$m$-isometric of normal WCT operators are coincide. Specially we have the results for multiplication operators. Indeed, we find that for $m\geq 2$, a multiplication operator $M_u$ is $m$-isometric (quasi-$m$-isometric) if and only if it is isometric (quasi-isometric). Some examples are provided to illustrate our results.

\end{abstract}

\maketitle

\section{ \sc\bf Introduction and Preliminaries}
Let $\mathcal{B}(\mathcal{H})$ be the Banach algebra of bounded linear operators acting on an infinite
dimensional complex Hilbert space $\mathcal{H}$. An operator $T\in \mathcal{B}(\mathcal{H})$ is called $m$-isometry if 
$$B_m=\sum_{k=0}^{m}(-1)^{m-k}\binom{m}{k}T^{*^{k}}T^{k}=0$$
and also $T$ is called quasi-m-isometry if
 $$\sum_{k=0}^{m}(-1)^{m-k}\binom{m}{k}T^{*^{k+1}}T^{k+1}=T^*B_mT=0.$$ 
 Specially, the operator $T$ is quasi-2-isometry if 
$$T^{*^3}T^3-2T^{*^2}T^2+T^*T=0.$$

Also, $T$ is called quasi-isometry if $T^{*^2}T^2=T^*T$. Clearly every
quasi-2-isometric operator is quasi-$m$-isometric (\cite{mp1}, Theorem 2.4). Moreover, the
classes of quasi-isometric operators and 2-isometric operators are sub-classes of the
class of quasi-2-isometric operators. In fact, inclusions are proper. There are some examples of quasi-2-isometric operators which are not quasi-isometric, one can see (\cite{mp1}, Example 1.1). For more details on $m$-isometric operators one can see \cite{as3, ber, bj,bm,cot,jl,mp} and references therein.\\
In this paper we are going to characterize $m$-isometric and quasi-$m$-isometric weighted composition operators on the Hilbert space $L^2(\mu)$.

\section{ \sc\bf Main Results}

Let $(X,\Sigma,\mu)$ be a complete $\sigma$-finite measure space.
For any sub-$\sigma$-finite algebra $\mathcal{A}\subseteq
 \Sigma$, the $L^2$-space
$L^2(X,\mathcal{A},\mu_{\mid_{\mathcal{A}}})$ is abbreviated  by
$L^2(\mathcal{A})$, and its norm is denoted by $\|.\|_2$. All
comparisons between two functions or two sets are to be
interpreted as holding up to a $\mu$-null set. The support of a
measurable function $f$ is defined as $S(f)=\{x\in X; f(x)\neq
0\}$. We denote the vector space of all equivalence classes of
almost everywhere finite valued measurable functions on $X$ by
$L^0(\Sigma)$.

\vspace*{0.3cm} For a sub-$\sigma$-finite algebra
$\mathcal{A}\subseteq\Sigma$, the conditional expectation operator
associated with $\mathcal{A}$ is the mapping $f\rightarrow
E^{\mathcal{A}}f$, defined for all non-negative measurable
functions $f$ as well as for all $f\in L^2(\Sigma)$, where
$E^{\mathcal{A}}f$, by the Radon-Nikodym theorem, is the unique
$\mathcal{A}$-measurable function satisfying
$\int_{A}fd\mu=\int_{A}E^{\mathcal{A}}fd\mu, \ \ \ \forall A\in
\mathcal{A}$. As an operator on $L^{2}({\Sigma})$,
$E^{\mathcal{A}}$ is idempotent and
$E^{\mathcal{A}}(L^2(\Sigma))=L^2(\mathcal{A})$. This operator
will play a major role in our work. Let $f\in L^0(\Sigma)$, then
$f$ is said to be conditionable with respect to $E$ if
$f\in\mathcal{D}(E):=\{g\in L^0(\Sigma): E(|g|)\in
L^0(\mathcal{A})\}$. We
write $E(f)$ in place of $E^{\mathcal{A}}(f)$.\\
Here we recall some basic properties of the conditional expectation operator $E$ on Hilbert space $L^2(\mu)$. Let $f,g\in L^2(\mu)$. Then\\
\vspace*{0.2cm} \noindent $\bullet$ \  If $g$ is
$\mathcal{A}$-measurable, then $E(fg)=E(f)g$.

\noindent $\bullet$ \ $|E(f)|^2\leq E(|f|^2)$.

\noindent $\bullet$ \ If $f\geq 0$, then $E(f)\geq 0$; if $f>0$,
then $E(f)>0$.

\noindent $\bullet$ \ $|E(fg)|^2\leq
E(|f|^2)E(|g|^{2})$, \ (conditional type H\"{o}lder inequality).\\

A detailed
discussion about this operator may be found in \cite{rao}. \\
\begin{defn}
Let $(X,\Sigma,\mu)$ be a $\sigma$-finite measure space and let $\mathcal{A}$ be a
$\sigma$-subalgebra of $\Sigma$, such that $(X,\mathcal{A},\mu)$ is also $\sigma$-finite. Suppose that $E$ is the corresponding conditional
expectation operator on $L^2(\mu)$ relative to $\mathcal{A}$. If $w,u \in L^0(\mu)$ such that $uf$ is conditionable and $wE(uf)\in L^2(\mu)$, for all $f\in L^2(\mu)$, then the corresponding weighted conditional type(WCT) operator is the linear
transformation $M_wEM_u:L^2(\mu)\rightarrow L^2(\mu)$ defined by $f\rightarrow wE(uf)$.\\
\end{defn}
Interested readers can find more information about WCT operators in \cite{e2,e,e1,ej,her,dhd}.

Let $\mathcal{H}$ be the infinite dimensional complex Hilbert
space, let $\mathcal{B(H)}$ be the algebra of all bounded
operators on $\mathcal{H}$. If $T\in \mathcal{B(H)}$, then $\sigma(T)$ and $r(T)$ are the spectrum and spectral radius of $T$, respectively. 
The operator $T\in \mathcal{B(H)}$ is called $p$-hyponormal if $(T^{\ast}T)^p\geq
(TT^{\ast})^p$, for $0<p<\infty$. As is known in the literature, if $T\in \mathcal{B(H)}$ is normal, i.e., $T^*T=TT^*$, then $r(T)=\|T\|$.\\

Here we recall from \cite{e1}, Theorem 2.8, that $$\sigma(M_wEM_u)\setminus \{0\}=ess \ range(E(uw))\setminus \{0\}.$$
If we set $S=S(E(|u|^2))$ and $G=S(E(|w|^2)$, then by conditional type Holder inequality we have $S(M_wEM_u(f))\subseteq S\cap G$, for all $f\in L^2(\mu)$.
Now by using basic properties of WCT operators, we characterize $m$-isometric and quasi-$m$-isometric WCT operators on $L^2(\mu)$.

\begin{thm}\label{t1.6}
Let $T=M_wEM_u$ be a bounded operator on $L^2(\mu)$ and $m\in \mathbb{N}$. Then
\begin{enumerate}
  \item WCT operator $T$ is quasi-$m$-isometry if and only if $|E(uw)|=1$, $\mu$, a.e. Consequently, we get that $T$ is quasi-isometry and only if it is quasi-$m$-isometry, for all $m\in \mathbb{N}$.\\
    
  \item WCT operator $T$ is $m$-isometry if and only if $E_r=\{1\}$, for odd $m$ and $E_r=\{-1\}$, for even $m$, in which 
    $$E_r=\text{essential\ range}(J'_mE(|w|^2)E(|u|^2))\ \ \ \text{and} \ \ \ J'_m=\sum_{k=1}^{m}(-1)^{m-k}\binom{m}{k}|E(uw)|^{2(k-1)}.$$
    Equivalently, $T$ is $m$-isometry if and only if $|J'_m|E(|w|^2)E(|u|^2)=1$, $\mu$, a.e.
  
\end{enumerate}
\end{thm}
\begin{proof}
(1) 
As we know, for each $k\in \mathbb{N}$, we have 
$$T^k=M_{E(uw)^{k-1}}T, \ \ T^{*^k}=M_{\overline{E(uw)}^{k-1}}T^* \ \ T^*T=M_{E(|w|^2)}M_{\bar{u}}EM_u$$

$$T^{*^k}T^k=M_{|E(uw)|^{2(k-1)}E(|w|^2)}M_{\bar{u}}EM_u.$$

Hence $T=M_wEM_u$ is quasi-$m$-isometry if and only if
$$\sum_{k=0}^{m}(-1)^{m-k}\binom{m}{k}T^{*^{k+1}}T^{k+1}=\sum_{k=0}^{m}(-1)^{m-k}\binom{m}{k}M_{|E(uw)|^{2k}}T^*T=0$$
if and only if $ M_{J_mE(|w|^2)\bar{u}}EM_u=0$
if and only if $$\|J_mE(|u|^2)E(|w|^2)\|_{\infty}=\|M_{J_k\bar{u}}EM_u\|=0,$$
in which $$J_m=\sum_{k=0}^{m}(-1)^{m-k}\binom{m}{k}|E(uw)|^{2k}.$$ 
Clearly, for every $x\in X$ we have $J_m(x)=(|E(uw)|(x)-1)^m$. 
So $T$ is quasi-$m$-isometry if and only if $J_m=0$, $\mu$, a.e., on $S(E(|u|^2))\cap S(E(|w|^2))$ if and only if $|E(uw)|=1$, $\mu$, a.e.\\ 

(2) By definition, $T$ is $m$-isometry if and only if
$$\sum_{k=0}^{m}(-1)^{m-k}\binom{m}{k}T^{*^{k}}T^{k}=((-1)^mI+\sum_{k=1}^{m}(-1)^{m-k}\binom{m}{k}M_{|E(uw)|^{2(k-1)}}T^*T)=0$$
if and only if
$$((-1)^mI+M_{J'_m}M_{E(|w|^2)\bar{u}}EM_u)=0$$
 
if and only if 
$$\|((-1)^mI+M_{J'_m}M_{E(|w|^2)\bar{u}}EM_u)\|=0,$$ 
in which 
$J'_m=\sum_{k=1}^{m}(-1)^{m-k}\binom{m}{k}|E(uw)|^{2(k-1)}$.
 It is easy to see that $M_{J'_m}M_{E(|w|^2)\bar{u}}EM_u$ is a self adjoint operator and so we get that
 
$$\|((-1)^mI+M_{J'_m}M_{E(|w|^2)\bar{u}}EM_u)\|^2=\|((-1)^mI+M_{J'_m}M_{E(|w|^2)\bar{u}}EM_u)^2\|.$$

This implies that 

\begin{align*}
\|((-1)^mI+M_{J'_m}M_{E(|w|^2)\bar{u}}EM_u)\|&=r((-1)^mI+M_{J'_m}M_{E(|w|^2)\bar{u}}EM_u)\\
&=\sup_{\lambda\in E}|(-1)^m+\lambda|,
\end{align*}
 in which 
 $$E_r=\sigma(M_{J'_m}M_{E(|w|^2)\bar{u}}EM_u)=\text{ess \ range}(J'_mE(|w|^2)E(|u|^2)).$$
 
 Therefore $T$ is $m$-isometry if and only if $\sup_{\lambda\in E_r}|(-1)^m+\lambda|=0$ if and only if $E_r=\{1\}$, for odd $m$ and $E_r=\{-1\}$, for even $m$. Equivalently, $T$ is $m$-isometry if and only if $|J'_m|E(|w|^2)E(|u|^2)=1$, $\mu$, a.e.
\end{proof}

Taking $E=I$ and $w=1$ in Theorem $\ref{t1.6}$ we have conditions under which the multiplication operator $M_u$ is quasi-$m$-isometry or $m$-isometry.
 
\begin{cor}
The followings hold for multiplication operator $M_u$ on $L^2(\mu)$.

\begin{itemize}
 \item  
The multiplication operator $M_u$ is $m$-isometry if and only if 
$$(|u(x)|^2-1)^m=\sum_{k=0}^{m}(-1)^{m-k}\binom{m}{k}|u(x)|^{2k}=0,$$
or equivalently $|u|=1$, $\mu$- a.e,.\\

\item Multiplication operator $M_u$ is quasi-$m$-isometry if and only if
$$|u(x)|^2(|u(x)|^2-1)^m=\sum_{k=0}^{m}(-1)^{m-k}\binom{m}{k}|u(x)|^{2(k+1)}=0,$$
 for almost all $x\in X$, or equivalently $|u|(1-|u|^2)=0$, $\mu$-a.e. 
 \end{itemize}
\end{cor}

Let $T\in \mathcal{B}(\mathcal{H})$ be a normal operator, i.e., $T^*T=TT^*$. Then for every $k\in \mathbb{N}$ we have $T^{*^k}T^k=(T^*T)^k$.
So we have 

$$B_m(T)=\sum_{k=0}^{m}(-1)^{m-k}\binom{m}{k}(T^*T)^{k}=(T^*T-I)^m.$$

It is clear that if $T$ is normal and $m$-isometry, for some $m\in \mathbb{N}$, then it is isometry and consequently it is unitary. Moreover, if $T$ is $p$-hyponormal and $T^n$ is normal for some $n$, then $T$ is normal. Therefore if $T$ is $p$-hyponormal, $T^n$ is normal for some $n$ and $T$ is $m$-isometric, then $T$ is unitary. In the next proposition we find that in a large class of bounded linear operators on the Hilbert space $L^2(\mu)$, the subclasses of hyponormal, $p$-hyponormal and normal operators are coincide. 
\begin{prop}\label{p2.4}
Let $T=M_wEM_u$ be the bounded WCT operators on the Hilbert space $L^2(\mu)$. Then $T$ is hyponormal if and only if it is $p$-hyponormal if and only if it is normal. 
\end{prop}
\begin{proof}
 As is proved in \cite{e}, Theorem 3.2, the WCT operator $M_wEM_u\in\mathcal{B}(\mathcal{H})$ is normal if and only if it is $p$-hyponormal. In addition, by Theorem 3.4 part (a) of \cite{ej}, we have $T$ is hyponormal if and only if it is $p$-hyponormal. Therefore by these observations we get that $T$ is hyponormal if and only if it is $p$-hyponormal if and only if it is normal. 
\end{proof}
As we mentioned above, if a bounded linear operator $T\in\mathcal{B(H)}$ is normal, then
 $$B_m(T)=\sum_{k=0}^{m}(-1)^{m-k}\binom{m}{k}(T^*T)^{k}=(T^*T-I)^m.$$
 
 Since $T^*T-I$ is a self adjoint operator, then it is normal and so we have 
  $$\|(T^*T-I)^m\|=r((T^*T-I)^m)=(r(T^*T-I))^m=\|T^*T-I\|^m.$$
  
  Therefore $T$ is isometry if and only if it is $m$-isometry for some $m\in \mathbb{N}$. In the next Theorem we prove that if the the WCT operator $T=M_wEM_u$ is $p$-hyponormal, then $T$ is isometry if and only if it is $m$-isometry for some $m\in \mathbb{N}$.
\begin{thm}
If $T=M_wEM_u$ is $p$-hyponormal, then the followings hold:\\

(a) The operator $T$ is isometric if and only if it is $m$-isometric for some $m\in \mathbb{N}$ if and only if $E(|w|^2)E(|u|^2)=1$, $\mu$-a.e.\\ 

(b) The operator $T$ is quasi-isometric if and only if it is quasi-$m$-isometric for some $m\in \mathbb{N}$ if and only if $E(|w|^2)E(|u|^2)=1$, $\mu$-a.e.\\ 

\end{thm}
\begin{proof}
(a) By Theorem \ref{t1.6} part (2), $T$ is $m$-isometry if and only if

$$B_m(T)=\sum_{k=0}^{m}(-1)^{m-k}\binom{m}{k}T^{*^{k}}T^{k}=((-1)^mI+\sum_{k=1}^{m}(-1)^{m-k}\binom{m}{k}M_{|E(uw)|^{2(k-1)}}T^*T)=0$$

if and only if
 
$$\|((-1)^mI+M_{J'_m}M_{E(|w|^2)\bar{u}}EM_u)\|=0,$$ 

in which  $J'_m=\sum_{k=1}^{m}(-1)^{m-k}\binom{m}{k}|E(uw)|^{2(k-1)}$. Since $T$ is normal, then by Corollary 2.15 of \cite{emj} we have $E(|w|^2)E(|u|^2)=|E(uw)|^2$. So direct computations show that 

$$J''_m=J'_mE(|w|^2)E(|u|^2)=(|E(uw)|^2-1)^m-(-1)^m.$$

 Thus we have
\begin{align*}
\|((-1)^mI+M_{J''_m}M_{\frac{\chi_S}{E(|u|^2)}\bar{u}}EM_u)\|&=r((-1)^mI+M_{J''_m}M_{\frac{\chi_S}{E(|u|^2)}\bar{u}}EM_u)\\
&=\sup_{\lambda\in E}|(-1)^m+\lambda|,
\end{align*}
 in which $S=S(E(|u|^2))$ and
  $$E=\sigma(M_{J''_m}M_{\frac{\chi_S}{E(|u|^2)}\bar{u}}EM_u)=\text{ess \ range}(J''_m).$$
 
 Therefore $T$ is $m$-isometry if and only if $\sup_{\lambda\in E}|(-1)^m+\lambda|=0$ if and only if $E=\{1\}$, for odd $m$ and $E=\{-1\}$, for even $m$. Equivalently, $T$ is $m$-isometry if and only if $|J'_m|E(|w|^2)E(|u|^2)=1$, $\mu$, a.e.\\
 
 (b) If $T$ is $p$-hyponormal, then by Remark \ref{p2.4} we get that $T$ is normal and so 
 $$E(|w|^2)E(|u|^2)=|E(uw)|^2.$$
 Hence by Theorem \ref{t1.6} we get that $T$ is quasi-isometric if and only if it is quasi-$m$-isometric, for some $m\in \mathbb{N}$, if and only if $E(|w|^2)E(|u|^2)=|E(uw)|^2=1$, $\mu$-a.e. Therefore we have the result.\\

\end{proof}
In the next corollary we see that if the conditional type Holder inequality for $u,w$ is an equality, then we have a characterization for isometric, $m$-isometric, quasi-isometric and quasi-$m$-isometric WCT operators.  
\begin{cor}\label{c2.6}
 If $T=M_wEM_u$ and $E(|w|^2)E(|u|^2)=|E(uw)|^2$, then the followings are equivalent:
 \begin{itemize}
 \item  $T$ is isometric;
 
 \item $T$ is $m$-isometric for some $m\in \mathbb{N}$;
 
 \item $T$ is quasi-isometric;
 
 \item $T$ is quasi-$m$-isometric for some $m\in \mathbb{N}$;
 
 \item $E(|w|^2)E(|u|^2)=|E(uw)|^2=1$, $\mu$-a.e.
 \end{itemize}
\end{cor}

Finally, we provide some examples to illustrate our main results.

\begin{exam} 

(a) Let $X=[0,1]\times [0,1]$, $d\mu=dxdy$,
$\Sigma$ the Lebesgue subsets of $X$ and let
$\mathcal{A}=\{A\times [0,1]: A \ \mbox{is a Lebesgue set in} \
[0,1]\}$. Then, for each $f$ in $L^2(\Sigma)$, $(Ef)(x,
y)=\int_0^1f(x,t)dt$, which is independent of the second
coordinate. Now, if we take $u(x,y)=y^{\frac{x}{8}}$ and $w(x,
y)=\sqrt{(4+x)y}$, then $E(|u|^2)(x,y)=\frac{4}{4+x}$ and
$E(|w|^2)(x,y)=\frac{4+x}{2}$. So, $E(|u|^2)(x,y)E(|w|^2)(x,y)=2$
and $|E(uw)|^2(x,y)=64\frac{4+x}{(x+12)^2}$. Direct computations
shows that 
$$E(|u|^2)(x,y)E(|w|^2)(x,y)\leq|E(uw)|^2(x,y), \ \ \text{a.e,}$$
and also by conditional type Holder inequality we have 
$$E(|u|^2)(x,y)E(|w|^2)(x,y)\geq|E(uw)|^2(x,y), \ \ \ \text{a.e.},$$ 
Hence we get that 
$$E(|u|^2)(x,y)E(|w|^2)(x,y)=|E(uw)|^2(x,y),  \ \ \ \text{a.e.}, $$
Now by Corollary \ref{c2.6} we get that $T=M_wEM_u$ is not isometric (equivalently, $m$-isometric, quasi-isometric, quasi-$m$-isometric), because the equation 
$$E(|u|^2)(x,y)E(|w|^2)(x,y)=|E(uw)|^2(x,y)=1$$
just for $x=4$ and so  
$$E(|u|^2)(x,y)E(|w|^2)(x,y)=|E(uw)|^2(x,y)\neq 1,  \ \ \ \text{a.e.}$$

(b)  Let $X=\mathbb{N}$,
$\mathcal{G}=2^{\mathbb{N}}$ and let $\mu(\{x\})=pq^{x-1}$, for
each $x\in X$, $0\leq p\leq 1$ and $q=1-p$. Elementary
calculations show that $\mu$ is a probability measure on
$\mathcal{G}$. Let $\mathcal{A}$ be the $\sigma$-algebra generated
by the partition $B=\{X_1, X^{c}_1\}$ of $X$, where $X_1=\{3n:n\geq1\}$. So,
for every $f\in \mathcal{D}(E^{\mathcal{A}})$,

$$E^{\mathcal{A}}(f)=\alpha_1\chi_{X_1}+\alpha_2\chi_{X^c_1}$$
and direct computations show that

$$\alpha_1=\frac{\sum_{n\geq1}f(3n)pq^{3n-1}}{\sum_{n\geq1}pq^{3n-1}}$$
and
$$\alpha_2=\frac{\sum_{n\geq1}f(n)pq^{n-1}-\sum_{n\geq1}f(3n)pq^{3n-1}}{\sum_{n\geq1}pq^{n-1}-\sum_{n\geq1}pq^{3n-1}}.$$

If we set $w(n)=n$ and $u(n)=\frac{1}{n}$, then we have 
$$\alpha_1=1, \ \ \ \ \ \alpha_2=1.$$
and so 

$$E^{\mathcal{A}}(wu)(n)=\chi_{X_1}(n)+\chi_{X^c_1}(n)=1,\ \ \ \ \text{for all} \ n\in \mathbb{N}.$$
Therefore by Theorem \ref{t1.6} we get that $T=M_wEM_u$ is quasi-$m$-isometry on $L^2(\mathbb{N},2^{\mathbb{N}},\mu)$,  for all $m\in \mathbb{N}$.

\end{exam}

\end{document}